\documentclass{proc-l-modified}
\usepackage[thmnumglobal, eqnumglobal, sepnum, amscompat]{defaults4}

\newcommand{\chebnorm}[1]{\norm{#1}_\infty}
\newcommand{\cnnorm}[2]{\chebnorm{{#1}^{(#2)}}}
\newcommand{\I}{[-1, 1]}
\newcommand{\Cn}{C^{\,n + 1}}
\newcommand{\Cno}{\Cn(\I)}
\newcommand{\Cm}{C^{\,m}}

\newcommand{\chebpoly}{T_{n+1}}

\title[Degree-$(n+1)$ poly. are difficult to approximate]%
{The degree-$(n+1)$ polynomials are the most difficult $C^{\,n + 1}$ functions to uniformly approximate with degree-$n$ polynomials.}
\author{Patrick Kidger}
\address{Mathematical Institute, University of Oxford, Andrew Wiles Building, Radcliffe Observatory Quarter (550), Woodstock Road, Oxford, OX2 6GG, UK}
\email{kidger@maths.ox.ac.uk}
\date{}



\thanks{This work was supported by the Engineering and Physical Sciences Research Council [EP/L015811/1]. The proof of Lemma \ref{lemma1} was slightly simplified by an anonymous commenter.}

\begin{document}

\begin{abstract}
There exist well-known tight bounds on the error between a function $f \in C^{\,n + 1}([-1, 1])$ and its best polynomial approximation of degree $n$. We show that the error meets these bounds when and only when $f$ is a polynomial of degree $n + 1$.
\end{abstract}

\maketitle

\begin{notation*}
Let $\Pi_n$ denote the set of polynomials of degree less than or equal to $n$.
\end{notation*}

It is a well known fact \cite{bernstein, phillips, elliottphillips, lewanowicz} that for all $f \in \Cno$ that
\begin{equation}\label{standard-bound}
\inf_{p \in \Pi_n} \chebnorm{f - p} \leq \frac{\cnnorm{f}{n+1}}{2^n (n+1)!}.
\end{equation}
(Note that this bound is better than that given by a na{\"i}ve Taylor series expansion.) When do we get equality? The purpose of this note is to show the following result.

\begin{thm*}
Let $n \in \naturals$. Let
\begin{equation*}
H_{n + 1} = \set{f \in \Cno}{\inf_{p \in \Pi_n} \chebnorm{f - p} = \frac{\cnnorm{f}{n+1}}{2^n (n+1)!}}.
\end{equation*}
Then $H_{n + 1} = \Pi_{n+1}$.
\end{thm*}

We highlight three particular implications of this result.

\begin{rem}
Every polynomial saturates an inequality of the form of equation \eqref{standard-bound}, for suitable $n$, whilst every (sufficiently differentiable) nonpolynomial fails to saturate any of them.
\end{rem}

\begin{rem}\label{rem-one}
Consider all target functions $f \in \Cn$ with fixed $\Cn$ seminorm, and thus by equation \eqref{standard-bound} of fixed maximum error. Then the Theorem shows that those $f$ which are also in $\Pi_{n+1}$ are precisely the $f$ which are worst approximated by elements of $\Pi_n$; hence the title of this note.
\end{rem}

\begin{rem}
It is typical to treat $\left(\Pi_n\right)_{n \in \naturals}$ as a nested sequence of improving approximations to the space of smooth functions. Then the Theorem implies that at every step the previous approximation scheme has been greedily improved by including those functions which it previously found most difficult to approximate.
\end{rem}

\begin{lem}\label{lemma1}
Let $m \in \naturals$. Let $z_0 < z_1 < \cdots < z_m$ be distinct points in $\reals$. Suppose $h$ is $m$ times differentiable in $[z_0, z_m]$, with $h(z_i) = 0$ for all $i$. Suppose also that $h^{(m)} \leq 0$ (or equivalently $h^{(m)} \geq 0$). Then $h \equiv 0$ in $[z_0, z_m]$.
\end{lem}

\begin{proof}
By induction. First consider $m = 1$. Then $h' \leq 0$ implies $h$ is nonincreasing, so $0 = h(z_0) \geq h(z) \geq h(z_1) = 0$ for all $z \in [z_0, z_1]$.

Now suppose the statement is true for $m - 1$, and consider the problem for $m$. Rolle's theorem implies for $i \in \{0, \ldots, m - 1\}$ that there exists $\eta_i \in (z_i, z_{i + 1})$ such that $h'(\eta_i) = 0$. By the inductive hypothesis $h' \equiv 0$ in $[\eta_0, \eta_{m-1}]$. Hence $h$ is constant there. So $h \equiv 0$ in $[\eta_0, \eta_{m-1}] \supseteq (\eta_0, \eta_{m-1}) \supseteq [z_1, z_{m-1}]$, and furthermore $h^{(m-1)} \equiv 0$ in $[z_1, z_{m-1}]$.

Now $h^{(m)} \leq 0$ implies $h^{(m - 1)}$ is nonincreasing, so $h^{(m - 1)} \geq 0$ in $[z_0, z_1]$ and $h^{(m - 1)} \leq 0$ in $[z_{m-1}, z_m]$. By the inductive hypothesis, $h \equiv 0$ in $[z_0, z_1]$ and $[z_{m-1}, z_m]$, and thus $h \equiv 0$ in $[z_0, z_m]$.
\end{proof}


\begin{lem}\label{lemma2}
Let $m \in \naturals$. Let $z_0 < z_1 < \ldots < z_m$ be distinct points in $\reals$. Let $\beta_0, \ldots, \beta_m \in \reals$. Let
\begin{equation*}
G = \set{g \in \Cm([z_0, z_m])}{g(z_i) = \beta_i \text{ for all } i}.
\end{equation*}
Let $p$ be the unique element of $G \cap \Pi_m$. Then $p$ is also the unique element of $G$ satisfying
\begin{equation*}
\cnnorm{p}{m} = \inf_{g \in G} \cnnorm{g}{m}.
\end{equation*}
\end{lem}
\begin{proof}
By considering $-p$ and $-\beta_i$ if necessary, also assume without loss of generality that $p^{(m)} \geq 0$, recalling that $p^{(m)}$ is constant. Let $g \in G$ be such that $\cnnorm{g}{m} \leq \cnnorm{p}{m}$, implying $g^{(m)} \leq p^{(m)}$, and let $h = g - p$. Then $h \equiv 0$ by Lemma \ref{lemma1}.
\end{proof}
That is, given values for some $m + 1$ points, then the unique polynomial of degree $m$ passing through them is also the unique smallest $\Cm$ function passing through them.

\begin{prop}[\cite{bernstein, phillips, elliottphillips}]\label{prop-one}
Let $f \in \Cno$. Then
\begin{equation*}
\frac{\min\limits_{x \in [-1, 1]}\abs{f^{(n+1)}(x)}}{2^n (n+1)!} \leq \inf_{p \in \Pi_n} \chebnorm{f - p} \leq \frac{\cnnorm{f}{n+1}}{2^n (n+1)!}.
\end{equation*}
\end{prop}

\begin{rem}
Consider all target functions $f \in \Cno$ for which
\begin{equation*}
\min_{x \in [-1, 1]}\abs{f^{(n+1)}(x)}
\end{equation*}
is of some fixed value, and so by Proposition \ref{prop-one} of fixed \emph{minimum} error. Then it is an immediate consequence of Proposition \ref{prop-one} that those $f$ which are also in $\Pi_{n+1}$ are among the $f$ which are \emph{best} approximated by elements of $\Pi_n$, in direct contrast to Remark \ref{rem-one}, and indeed the title of this article.

This is of course because of the dependence on how the size of a $\Cn$ function was fixed. The notion of \emph{best} corresponds to the use of
$f \mapsto \min_{x \in [-1, 1]}\abs{f^{(n+1)}(x)}$
as a measure of the size of a $\Cn$ function, whilst \emph{worst} corresponds to the more typical $\Cn$ seminorm.
\end{rem}

\begin{notation*}
Let $\chebpoly$ denote the $(n + 1)$-th Chebyshev polynomial of the first kind. Let $x_0, \ldots x_n \in \I$ be the roots of $\chebpoly$, so that
\begin{equation*}
x_i = \cos(\frac{(2i + 1) \pi }{2n + 2}).
\end{equation*}
\end{notation*}

\begin{prop}[\cite{phillips}]\label{prop-two}
Let $f \in \Cno$, and let $q$ denote the degree-$n$ polynomial interpolating $f$ through $x_0, \ldots, x_n$. Then for all $x \in \I$ there exists $\zeta_x \in \I$ such that
\begin{equation*}
f(x) - q(x) = \frac{f^{(n + 1)}(\zeta_x) \chebpoly (x)}{2^n (n+1)!}.
\end{equation*}
\end{prop}

\begin{proof}[Proof of Theorem]
The forward inclusion is straightforward; it follows immediately from Proposition \ref{prop-one} that $\Pi_{n+1} \subseteq H_{n + 1}$.

Now the reverse inclusion. Let $f \in H_{n + 1}$. If $f \in \Pi_n \subseteq \Pi_{n + 1}$ we are done, so assume $f \notin \Pi_n$. As $f \in \Cno$, then by Proposition \ref{prop-two},
\begin{equation}\label{some-eq-label}
\chebnorm{f - q} \leq \frac{\cnnorm{f}{n+1}}{2^n (n+1)!}.
\end{equation}

Then equation \eqref{some-eq-label} and the fact that $f \in H_{n + 1}$ give that
\begin{equation}\label{q-min}
\chebnorm{f - q} \leq \frac{\cnnorm{f}{n+1}}{2^n (n+1)!} = \inf_{p \in \Pi_n} \chebnorm{f - p} \leq \chebnorm{f - q}.
\end{equation}
That is, $q$ achieves the infimum; it is the minimax approximation.

That $q$ is the minimax approximation implies, by the Equioscillation Theorem \cite[Theorem 10.1]{trefethen}, that $\abs{f - q}$ achieves its maximum at some $n + 2$ distinct points $-1 \leq y_0 < y_1 < \ldots < y_{n+1} \leq 1$, for which
\begin{equation}\label{another-eq-label}
f(y_j) - q(y_j) = \sigma (-1)^j \chebnorm{f - q},
\end{equation}
where  $\sigma \in \{-1, 1\}$. Together with Proposition \ref{prop-two} and equation \eqref{q-min} this implies that
\begin{equation}
\frac{f^{(n + 1)}(\zeta_{y_j}) \chebpoly (y_j)}{2^n (n+1)!} = \sigma (-1)^j \frac{\cnnorm{f}{n+1}}{2^n (n+1)!}. \label{yj-max}
\end{equation}
So the $y_j$ are precisely the locations of the maxima and minima of $\chebpoly$ in $[-1, 1]$; in particular $y_0 = -1$ and $y_{n+1} = 1$.

Let $r = f -q$, and without loss of generality assume $f$ is normalised such that
\begin{equation}\label{eq:norm}
\cnnorm{f}{n+1} = 2^n (n + 1)!,
\end{equation}
which is possible as $f \notin \Pi_n$.

Assume also without loss of generality that $\sigma = -1$ in equation \eqref{yj-max}; if need be swap $f$ for $-f$.

Then $r$ has the following properties:
\begin{enumerate}[label=(\roman*)]
\item $r \in \Cno$, as both $f$ and $q$ are.

\item $\cnnorm{r}{n+1} = 2^n (n + 1)!$, by equation \eqref{eq:norm}, because $q$ is a degree-$n$ polynomial and so vanishes when differentiated $n +1$ times.

\item $r(x_i) = 0 = \chebpoly(x_i)$ for all $x_i$, by Proposition \ref{prop-two}, recalling that $x_i$ are the zeros of $\chebpoly$.

\item $r(y_j) = (-1)^{j + 1} = \chebpoly(y_j)$ for all $y_j$. The first equality follows by substituting equation \eqref{eq:norm} into equation \eqref{q-min}, and then substituting the result into equation \eqref{another-eq-label}. The second equality is because the $y_j$ are the locations of the maxima and minima of $\chebpoly$, which followed from equation \eqref{yj-max}.
\end{enumerate}

Let $\{z_0, \ldots, z_{n+1}\}$ be any $n + 2$ distinct points of $\{x_0, \ldots, x_n, y_0, \ldots, y_{n + 1}\}$, such that in particular $z_0 = y_0 = -1$ and $z_{n + 1} = y_{n + 1} = 1$. Let $\beta_i = r(z_i)$. Then by Lemma \ref{lemma2}, there exists a unique function $p \in \Cno$ mapping $z_i$ to $\beta_i$ with minimal $\cnnorm{p}{n + 1}$, and furthermore $p \in \Pi_{n + 1}$.

Two uniqueness results are now used to identify $p$ in two different ways.

First, we observe that $\chebpoly \in \Pi_{n+1}$ maps the $n+2$ points $z_i$ to $\beta_i$, by properties (iii) and (iv), and so $p = \chebpoly$ by uniqueness of polynomial interpolants. This also implies that
\begin{equation}\label{eq-label-blah}
\inf_{g \in G} \cnnorm{g}{n + 1} = \cnnorm{p}{n + 1} = \cnnorm{\chebpoly}{n + 1} = 2^n (n + 1)!,
\end{equation}
where $G$ is as in Lemma \ref{lemma2}.

Second, Lemma \ref{lemma2} also states that $p$ is the unique element of $G$ that attains $\inf_{g \in G} \cnnorm{g}{n + 1}$. But $r \in G$, by properties (i), (iii), (iv), and it attains the minimum
\begin{equation*}
\inf_{g \in G} \cnnorm{g}{n + 1} = 2^n (n + 1)! = \cnnorm{r}{n + 1},
\end{equation*}
by equation \eqref{eq-label-blah} and property (ii), and so in fact $p = r$ also.

Thus $r = \chebpoly$. And so $f = q + r \in \Pi_{n+1}$ as desired. (Noting also that the normalisation and changing of sign are valid in that the `original' $f$ must also belong to $\Pi_{n+1}$.)
\end{proof}

\end{document}